\documentclass[12pt]{article}
\usepackage{fullpage}
\usepackage{amsmath, amssymb, latexsym, amsthm, graphicx, color}
\theoremstyle{plain}
\makeatletter
\newtheorem{thm}{Theorem} %[section]
\newtheorem{cor}[thm]{Corollary}
\newtheorem{lem}[thm]{Lemma}

\theoremstyle{remark}

\theoremstyle{definition}

\def\revs{revolutionaries}
\def\rev{revolutionary}
\def\floor#1{\left\lfloor #1 \right\rfloor}

\def\FL#1{\left\lfloor #1 \right\rfloor}
\def\CL#1{\left\lceil #1 \right\rceil}
\def\FR#1#2{{\frac{#1}{#2}}}
\def\RSG{{\rm RS}(G,m,r,s)}
\def\RS{{\rm RS}}
\def\sgmr{\sigma(G,m,r)}
\def\sgtr{\sigma(G,2,r)}
\def\sghr{\sigma(G,3,r)}
\def\SM#1#2{\sum_{#1\in #2}}
\def\st{\colon\,}
\def\NN{{\mathbb N}}
\def\RR{{\mathbb R}}

\begin{document}

\title{Revolutionaries and spies on trees and unicyclic graphs}

\author{
Daniel W. Cranston\thanks{Virginia Commonwealth University,
dcranston@vcu.edu.}\,,
Clifford D. Smyth\thanks{University of North Carolina -- Greensboro,
cdsmyth@uncg.edu.}\,,
Douglas B. West\thanks{University of Illinois,
west@math.uiuc.edu, partially supported by NSA grant H98230-10-1-0363.}
}

\maketitle

\vspace{-2pc}

\begin{abstract}
A team of $r$ {\it \revs} and a team of $s$ {\it spies} play a game on a graph
$G$.  Initially, \revs\ and then spies take positions at vertices.  In each
subsequent round, each \rev\ may move to an adjacent vertex or not move, and
then each spy has the same option.  The \revs\ want to hold an {\it unguarded
meeting}, meaning $m$ \revs\ at some vertex having no spy at the end of a
round.  To prevent this forever, trivially at least $\min\{|V(G)|,\FL{r/m}\}$
spies are needed.  When $G$ is a tree, this many spies suffices.  When $G$ is a
unicyclic graph, $\min\{|V(G)|,\CL{r/m}\}$ spies suffice, and we characterize
those unicyclic graphs where $\FL{r/m}+1$ spies are needed.
\end{abstract}

\baselineskip 16pt

\section{Introduction}

Many pursuit games have been studied on graphs.  We study such a game that can
be interpreted as modeling a problem of network security.  One team consists of
$r$ {\it\revs}; the other consists of $s$ spies.  The \revs\ want to arrange a
one-time meeting of $m$ \revs\ free of oversight by spies.  Initially, the
\revs\ take positions at vertices, and then the spies do the same.  In each
subsequent round, each \rev\ may move to an adjacent vertex or not move, and
then each spy has the same option.  Everyone knows where everyone else is.

The \revs\ win if at the end of a round there is an {\it unguarded meeting},
where a {\it meeting} is a set of (at least) $m$ \revs\ on one vertex, and a
meeting is {\it unguarded} if there is no spy at that vertex.  The spies win if
they can prevent this forever.  Let $\RSG$ denote this game played on the graph
$G$ by $s$ spies and $r$ \revs\ seeking an unguarded meeting of size $m$.

The \revs\ can form $\min\{|V(G)|,\FL{r/m}\}$ meetings initially; if $s$ is
smaller than this, then the spies immediately lose.  On the other hand, the
spies win if $s\ge r-m+1$; they follow $r-m+1$ distinct \revs, and the other
$m-1$ \revs\ cannot form a meeting.  For fixed $G,r,m$ we study the minimum $s$
such that the spies win $\RSG$; let $\sgmr$ denote this threshold.  The trivial
bounds are
$$\min\{|V(G)|,\FL{r/m}\}\le\sgmr\le\min\{|V(G)|,r-m+1\}.$$

The game was invented by Jozef Beck in the mid-1990s.  Shortly thereafter,
Smyth proved that the trivial lower bound on $\sgmr$ is sufficient when $G$ is
a tree.  This was not published; we include a proof here and use the result in
solving the game for unicyclic graphs.  Howard and Smyth~\cite{HS} studied
$\RS(G,2,r,s)$ when $G$ is the infinite $2$-dimensional integer grid with
one-step horizontal, vertical, and diagonal edges.  They proved
$6\FL{r/8}\le \sgtr\le r-2$.  For the upper bound, which they conjectured
is sharp, they showed that $r-2$ spies can win by having all but one follow
\revs; the one remaining spy can prevent the remaining three \revs\ from making
an unguarded meeting.

In this paper, we determine $\sgmr$ for all trees and unicyclic graphs.  For a
graph $G$ with $r/m<|V(G)|$, we show that $\sgmr\le\CL{r/m}$ if $G$ has at most
one cycle, by summing over components.  When $G$ is a tree, $\FL{r/m}$ spies
suffice, and $\CL{r/m}$ suffice when $G$ is connected and unicyclic.  Our final
result is that if $G$ is unicyclic with a cycle of length $\ell$ and $t$
vertices not on the cycle (connected or not), and $m\nmid r$, then
$\sgmr=\FL{r/m}$ if and only if $\ell\le \max\{\FL{r/m}-t+2,3\}$.

Butterfield, Cranston, Puleo, West, and Zamani~\cite{BCPWZ} study $\sgmr$ on
a variety of graphs.  With $r/m<|V(G)|$, they show that $\sgmr= \FL{r/m}$
when $G$ has a rooted spanning tree $T$ such that every edge of $G$ not in $T$
joins vertices having the same parent in $T$; this includes graphs with a
dominating vertex and interval graph.  For every graph $G$, they prove
$\sgmr\le \gamma(G)\FL{r/m}$, where $\gamma(G)$ is the domination number of $G$,
and this is nearly sharp: for $t,m,r\in\NN$ with $t\le m$, there is a graph
$G$ with domination number $t$ such that $\sgmr>t(r/m-1)$.  Also, there are
chordal graphs (and bipartite graphs) for given $r$ and $m$ such that
$\sgmr=r-m+1$.  If $G$ is the $d$-dimensional hypercube $Q_d$ with $d\ge r$,
then $\sgtr=r-1$; in general, if $d\ge r\ge m\ge 3$, then 
$\sigma(Q_d,m,r)>r-\FR34 m^2$.

They also consider $r$-large complete $k$-partite graphs, where ``$r$-large''
means that each partite set has size at least $2r$.  When $k\ge m$ and $G$ is
such a graph,
$$
\FR{k\cdot k\FL{r/k}}{(k-1)m+1}-k\le \sgmr\le \CL{\FR k{k-1}\FR rm}+k.
$$
When $G$ is a large bipartite graph, they show that $\sgtr=7r/10$ and
$\sghr=3r/2$ (within additive constants) and that in general
$$
\FR{\FL{r/2}}{\CL{m/3}}-2\le \sgmr\le \left(1+\FR1{\sqrt3}\right)\FR rm+1.
$$
The upper bound coefficient on $\FR rm$ is about $1.58$; the lower one is
about $1.5$ when $m$ is large and $r>m$; they conjecture that the lower one
is the true coefficient, at least when $3\mid m$.

\section{Trees and Cycles}

If $r/m<|V(G)|$, then the spies lose whenever $s<\FL{r/m}$.  Hence the first
chance for the spies is when $s=\FL{r/m}$.  We prove that this suffices when
$G$ is a tree, yielding $\sgmr=\FL{r/m}$.  Since the spies always win when
$s\ge|V(G)|$, the statement of our first theorem remains true regardless of
the relationship between $r/m$ and $|V(G)|$.  Nevertheless, {\bf to avoid
trivial statements, we henceforth assume always that $r/m\le|V(G)|$}.

\begin{thm}\label{tree}
If $G$ is a tree and $s\ge\floor{r/m}$, then the spies win $\RSG$.
\end{thm}
\begin{proof}
It suffices to show that the spies win when $s=\floor{r/m}<|V(G)|$.  Choose a
root $z\in V(G)$.  The {\it parent} of a non-root vertex $v$, denoted $v^+$, is
its neighbor on the path from $v$ to $z$.  The other neighbors of $v$ are its
{\it children}; let $C(v)$ be this set of children of $V$.
The {\it descendants} of $v$ are all the vertices (including $v$) whose path to
$z$ contains $v$; let $D(v)$ be the set of descendants of $v$.

For each vertex $v$, let $r(v)$ and $s(v)$ denote the number of \revs\ and
spies on $v$ at the current time, respectively, and let $w(v)=\SM u{D(v)}r(u)$.
The spies maintain the following invariant giving the number of spies on
each vertex at the end of any round:
\begin{equation}\label{invar}
s(v)=\floor{\frac{w(v)}{m}}-\sum_{u\in C(v)} \floor{\frac{w(u)}{m}}
\qquad\textrm{for }v\in V(G).
\end{equation}
Since $\SM u{C(v)}w(u)=w(v)-r(v)$, the formula is always nonnegative.  Also, if
$r(v)\ge m$, then $s(v)\ge \FL{\FR{w(v)}m}-\FL{\FR{w(v)-r(v)}m}\ge 1$.  Hence
(\ref{invar}) guarantees that every meeting is guarded.

To show that the spies can establish (\ref{invar}) after the first round, it
suffices that all the formulas sum to $\floor{r/m}$.  More generally, summing
over the descendants of any vertex $v$,
\begin{equation}\label{telsum}
\SM u{D(v)} s(u)= \FL{\FR{w(v)}m},
\end{equation}
since $\FL{w(u)/m}$ occurs positively in the term for $u$ and negatively in the
term for $u^+$, except that $\FL{w(v)/m}$ occurs only positively.  When $v=z$,
the total is $\FL{r/m}$, since $w(z)=r$.

To show that the spies can maintain (\ref{invar}), let $r'(v)$ denote the new
number of \revs\ at $v$ after the \revs\ move, and let $w'(v)=\SM u{D(v)}r'(v)$.
The spies move to achieve the new values required by (\ref{invar}), starting
from the leaves; we will produce $s'(v)$ at $v$ after adjusting at all the
children (and lower descendants) of $v$.

We process siblings simultaneously.  That is, having updated all children of
$C(v)$, we adjust at all of $C(v)$ simultaneously.  In doing this, excess spies
move to $v$, and needed spies come from $v$; no changes are made involving
children of $C(v)$.  Similarly, fixing $C(v^+)$ later includes an exchange
between $v$ and $v^+$ but does not disturb the spies on $C(v)$.  After
successfully updating $C(v)$ for all $v$, the root vertex $z$ will have exactly
$s'(z)$ spies, since always $\sum s'(v)=\sum s(v)=\FL{r/m}$.

We can now process $C(v)$.  Let $D^*(u)=D(u)-\{u\}$ for all $u$.  For
$u\in C(v)$, all vertices in $D^*(u)$ have been adjusted.  Spies previously on
$D^*(u)$ remained in $D(u)$, and those now on $D^*(u)$ came from $D(u)$.  Hence
$u$ now has $\SM t{D(u)}s(t)-\SM t{D^*(u)}s'(t)$ spies.  Let
\begin{equation}\label{diff}
\partial(u)=s'(u)-\SM t{D(u)}s(t)+\SM t{D^*(u)}s'(t)
=\FL{\FR{w'(u)}m}-\FL{\FR{w(u)}m}.
\end{equation}
As defined, changing by $\partial(u)$ achieves $s'(u)$ spies at $u$; the second
equality uses (\ref{telsum}).  When $\partial(u)$ is positive, $\partial(u)$
spies move from $v$ to $u$; when it is negative, $u$ sends $-\partial(u)$ spies
to $v$.

Let $C^+=\{u\in C(v)\st w'(u)>w(u)\}$ and $C^-=C-C^+$.  Note that both $C^+$
and $C^-$ may contain vertices for whom the adjustment from or to $v$ is $0$.
To avoid making spies take two steps to reach $C^+$, we must ensure
$\SM u{C^+}\partial(u)\le s(v)$.  To avoid forcing spies from $C^-$ to take a
second step, we must ensure $\SM u{C^-}|\partial(u)|\le s'(v)$.

For the first inequality, note that $\sum_{u \in C^+} [w'(u)-w(u)] \le r(v)$,
since \revs\ who entered subtrees rooted at $C(v)$ on this round were
previously at $v$.  Thus
\begin{align*}
\SM u{C^+}\FL{\FR{w'(u)}m}&\le \FL{\SM u{C^+} \FR{w'(u)}m} \le
\FL{\FR{r(v)+\SM u{C^+} w(u)}m} = \FL{\FR{w(v)-\SM u{C^-} w(u)}m}\\
&\le \FL{\FR{w(v)}m}-\SM u{C^-}\FL{\FR{w(u)}m}
= s(v)+\SM u{C^+}\FL{\FR{w(u)}m}.
\end{align*}
By (\ref{diff}), this yields the desired inequality.
For the second inequality, $\sum_{u \in C^-} [w(u)-w'(u)] \le r'(v)$, since
\revs\ who left subtrees rooted at $C^-$ on this round are now at $v$.  Thus 
\begin{align*}
\SM u{C^-}\FL{\FR{w(u)}m}&\le \FL{\SM u{C^-} \FR{w(u)}m} \le
\FL{\FR{r'(v)+\SM u{C^-} w'(u)}m} = \FL{\FR{w'(v)-\SM u{C^+} w'(u)}m}\\
&\le \FL{\FR{w'(v)}m}-\SM u{C^+}\FL{\FR{w'(u)}m}
= s'(v)+\SM u{C^-}\FL{\FR{w'(u)}m}.
\end{align*}
Again (\ref{diff}) yields the desired inequality.

These inequalities ensure that the adjustments restore (\ref{invar}) by legal
moves.  The ad\-just\-ment to $C(v)$ is allowed, after which $v$ has
$s(v)-\SM u{C(v)}\partial(u)$ spies.  Using (\ref{invar}) and (\ref{diff}),
\begin{equation*}
s(v)-\SM u{C(v)}\partial(u)
%&=\FL{\FR{w(v)}m}-\SM u{C(v)}\FL{\FR{w(u)}m}
%+\SM u{C(v)}\FL{\FR{w(u)}m}-\SM u{C(v)}\FL{\FR{w'(u)}m}\\
=s'(v)-\FL{\FR{w'(v)}m}+\FL{\FR{w(v)}m}~=~ s'(v)-\partial(v).
\end{equation*}
Thus, when $C(v^+)$ is later processed, the adjustment needed at $v$ is exactly
what we have said will be made.  Furthermore, spies moving to $v$ from $C^-$ do
not come from below $C^-$ and do not continue on to $v^+$.  The argument for
the latter applies also to the former.  If $\partial(v)<0$, requiring
$-\partial(v)$ spies to move to $v^+$, then the adjustment at $C(v)$ left more
than $s'(v)$ spies at $v$.  However, since $\SM u{C^-}|\partial(u)|\le s'(v)$,
the extra spies beyond $s'(v)$ were among the $s(v)$ spies on $v$ at the
beginning of the round, so they can move to $v^+$.
\end{proof}

\begin{cor}\label{acyclic}
If $G$ is a forest and $s\ge\FL{r/m}$, then the spies win $\RSG$.
%That is, $\sgmr\le \min\{\FL{r/m},|V(G)|\}$.
\end{cor}

We next show that $\CL{r/m}$ spies suffice to win on a cycle.

\begin{lem}\label{cycle}
If $G$ is a cycle, then $\sgmr\le \CL{r/m}$.
Equality holds when $m\mid r$ and $r/m\le |V(G)|$.
\end{lem}
\begin{proof}
Extra \revs\ cannot make the game easier for the spies, so we may assume that
$r=sm$.  Given the initial locations of \revs, index the \revs\ in order from
$0$ to $sm-1$ around the cycle.  Place spies on the vertices occupied by the
\revs\ whose index is a multiple of $m$ (this may put more than one spy on a
vertex).

Because \revs\ are identical, we may assume that the \revs\ always 
remain indexed in order around the cycle as they move (equivalently, if \revs\
switch positions, then they trade indices).  Since indices move along at most
one edge in each round, the $i$th spy can continue to follow the \rev\ with
index $im$.  Thus after every round, any vertex occupied by at least $m$ \revs\
is guarded by at least one spy.
\end{proof}

When the cycle is short enough, the threshold for a spy win improves to the
trivial lower bound $\FL{r/m}$.  For longer cycles, the strategy for the \revs\
to defeat $\FL{r/m}$ spies may take many rounds to produce an unguarded meeting.
Note that when $r\mid m$, the upper bound from Lemma~\ref{cycle} coincides
with the trivial lower bound to yield $\sigma(C_n,m,r)=r/m$ when $n\ge r/m$.

\begin{thm}\label{fuzzy-cycle}
If $G$ is a cycle of length $\ell$, and $m\nmid r$ with $r/m\le\ell$, then
$\sgmr=\FL{r/m}$ if and only if $\ell\le \FL{r/m}+2$ or $r<m$.
\end{thm}
\begin{proof}
Let $s=\FL{r/m}$.  When $r<m$, no meetings can be formed, so no spies are
needed.  When $\ell\le s+2$, the spies can guard all initial meetings and leave
at most two vertices unguarded.  After a move by the \revs, at least two
vertices $u$ and $v$ fail to host a meeting.  The spies can move to leave only
$u$ and $v$ unguarded, by shifting one step along paths from $u$ and $v$ to the
previously unguarded vertices.

For the converse, consider $\ell\ge s+3$.  It suffices to show that the \revs\
win when $r=sm+1$.  They will first distract one spy $S$, arranging for $S$ to
guard a vertex occupied by at most one \rev.  They then win by making $s$
meetings on the remaining path, guarded by at most $s-1$ spies.

The \revs\ move so that no more than $m$ of them ever occupy one vertex.
Subject to this, they start with any initial distribution.  The spies take
initial positions, and the \revs\ designate one spy as $S$.  To reduce what $S$
guards, on each round the \revs\ guarded by $S$ move away, half in one
direction and half in the other (rounded to integers).  The \revs\ that were on
the neighboring vertices move farther away, but still each vertex has at most
$m$ \revs; this is possible since $\ell\ge s+3$.  No matter how $S$ moves, the
number of \revs\ guarded by $S$ is at most half (rounded up) of what it was
before.  After at most $\CL{\log_2 m}$ rounds, $S$ guards at most one \rev.

Now the \revs\ shorten the path containing the other \revs\ by moving those
nearest to $S$ away from $S$, maintaining that each vertex has at most $m$
\revs\ (again possible since $\ell\ge s+3$).  They continue until the path
consists of $s$ consecutive vertices with meetings, which cannot all be guarded
by the $s-1$ spies other than $S$.
\end{proof}

\section{Unicyclic Graphs}
A connected unicyclic graph has a cycle and trees attached to it.  The cycle
and trees interact, since \revs\ may move on and off the cycle.  The spies must
respond appropriately.

\begin{thm}\label{unicyclic}
If $G$ is a unicyclic graph, then $\sgmr\le \CL{r/m}$.
\end{thm}
\begin{proof}
It suffices to show that the spies win $\RSG$ when $s=\CL{r/m}\le|V(G)|$.
If there are $r_i$ \revs\ in a component $G_i$ of $G$, and $G_i$ is a tree,
then only $\FL{r_i/m}$ spies are needed in $G_i$.  Since at most one component
contains a cycle, and $\FL{a}+\CL{b}\le \CL{a+b}$ for all $a,b\in\RR$, we may
assume that $G$ is connected and contains a cycle $C$.  View $G-V(C)$ as
disjoint trees rooted at vertices neighboring $C$.  In order to use the
strategies of Theorem~\ref{tree} and Lemma~\ref{cycle}, a spy must be available
when needed to move from $C$ to a tree (or vice versa).

Say that the current state satisfies the {\it cycle condition} if the number of
\revs\ on $C$ is $m$ times the number of spies on $C$ and there is a spy
guarding every $m$th \rev\ as described in the proof of Lemma~\ref{cycle}.  The
key fact needed is that adding one spy and $m$ \revs\ to any vertex of $C$ (or
removing them) preserves the cycle condition.

As in Lemma~\ref{cycle}, we may assume $m\mid r$.  We may also assum that
all \revs\ start on the cycle.  To play against another initial position, the
spies imagine an initial position on the cycle and follow their winning
strategy as the \revs\ move to the actual start.  With all \revs\ initially on
the cycle, the cycle condition holds at the start. 

The {\it attached trees} are the components of $G-V(C)$.  The root of an
attached tree $T$ is the vertex $z$ adjacent to a vertex of $C$.  The neighbor
of $z$ in $V(C)$, denoted $z^*$, is the {\it mate} of $T$.  As \revs\ disperse
to or return from the attached trees, maintaining the cycle condition requires
keeping a buffer of ``fake'' \revs\ for each tree at its mate on $C$ (a vertex
may be the mate of many trees).

When a \rev\ moves from $C$ to an attached tree $T$, it moves to the root $z$
of $T$ from the mate $z^*$ on $C$.  Until $m$ \revs\ move to $T$, no spy needs
to follow, since $m-1$ \revs\ cannot make a meeting on $T$.  For each \rev\
that moves from $z^*$ to $z$, we add a fake \rev\ at $z^*$; this maintains
the cycle condition.  The fake \revs\ are markers maintained by the spies and 
do not move.

When $m$ actual \revs\ have moved to $T$, a spy is needed.  Before the final
move of $j$ \revs\ from $z^*$ to $z$, there were also $m-j$ fake \revs\ at
$z^*$, so by the cycle condition there was a spy at $z^*$.  This spy moves to
$z$ following the $j$ \revs\, and $m-j$ fake \revs\ disappear from $z^*$.  This
preserves the cycle condition.

In the strategy on $T$ given in Theorem~\ref{tree}, when \revs\ are added at
the root to increase the number of spies needed, the vertex needing the extra
spy is the root.  Hence the arrival of the spy at $z$ from $z^*$ permits the spy
strategy on $T$ to operate locally.  Similarly, when \revs\ leave $T$ to reduce
the number of spies needed, they leave the root.  The number of spies computed
for other vertices of $T$ does not change, so the location of the extra spy is
$z$.  It can return to $z^*$ and reestablish the appropriate number of fake
\revs.

As spies follow \revs\ onto and off the cycle, the fake \revs\ enable the spies
to maintain the cycle condition, and the strategies for spies on the cycle and
on the attached trees can operate independently as previously given.
\end{proof}

By Theorem~\ref{unicyclic}, $\sgmr\in\{\FL{r/m},\CL{r/m}\}$ when $G$ is
unicyclic and $|V(G)|\ge r/m$.  Theorem~\ref{unicyclic2} determines which is
the answer.  The role of vertices outside the cycle is shown by the disjoint
union of $C_5$ and $P_2$.  Three spies defeat seven \revs\ on $C_5$ when $m=2$,
by Theorem~\ref{fuzzy-cycle}.  However, four \revs\ can sit on $P_2$ forever,
occupying two spies, and the remaining three \revs\ defeat the remaining spy on
$C_5$.  Adding an edge to join the two components does not affect the \revs's
strategy; it does not matter whether the graph is connected.

\begin{thm}\label{unicyclic2}
Let $G$ be a unicyclic graph having a cycle $C$ of length $l$ and exactly $t$
vertices not on $C$.  If $|V(G)|\ge r/m$, then $\sgmr\in\{\FL{r/m},\CL{r/m}\}$,
equalling $\FL{r/m}$ when $m\nmid r$ if and only if
$\ell\le\max(\FL{r/m}-t+2,3)$ or $r<m$.
\end{thm}
\begin{proof}
By Theorem~\ref{unicyclic}, $\sgmr\le\CL{r/m}$, whether or not $G$ is connected.
Assume $m\nmid r$, and let $s=\FL{r/m}\le|V(G)|$.  It suffices to show that the
\revs\ win $\RSG$ if $\ell\ge\max\{s-t+3,4\}$, and the spies win $\RSG$ if
$\ell\le\max\{s-t+2,3\}$.

When $\ell\ge\max\{s-t+3,4\}$, the \revs\ first make meetings at $k$ vertices
outside $C$, where $k=\min\{t,s-1\}$.  These meetings must be guarded by $k$
spies.  These $km$ \revs\ (and thus also $k$ spies) will not move (it does not
matter whether their component contains $C$).  On $C$, the remaining $r-km$
\revs\ play against $s-k$ spies.  Since $r/m > s$, also $(r-km)/m > s-k$.
Since $s-k=\max\{s-t,1\}$, we have $\ell\ge s-k+3$.  By
Theorem~\ref{fuzzy-cycle}, the remaining \revs\ win on the cycle.

Now suppose $\ell\le\max\{s-t+2,3\}$.  If $G$ has a component $H$ with $\hat t$
vertices and $\hat r$ \revs\ outside the component containing $C$, then $H$ 
needs only $\min\{\hat t,\FL{\hat r/m}\}$ spies.  On $G-V(H)$, since
$\FL{a}+\FL{b}\le\FL{a+b}$, the specified conditions hold for the remaining
spies to win.  Hence to consider a minimal counterexample we may assume that
$G$ is connected.  As in Theorem~\ref{fuzzy-cycle}, we may also assume that all
\revs\ start on $C$.
%The idea is that when \revs\ draw spies onto the vertices outside $C$, the
%cycle is short enough that the remaining spies can win there
%against the remaining \revs.

\medskip
{\bf Case 1:} $s>t$.
In this case, $\ell\le s-t+2$.  Since $|V(G)|>\FL{r/m}=s$, in fact
$\ell\in\{s-t+1,s-t+2\}$.  Each attached tree $T$ reserves $|V(T)|$ spies,
located initially at the mate of $T$.  These ``tree spies'' mostly remain at
the mate of $T$ except to move into $T$ as needed to play the spy strategy on
$T$ from Theorem~\ref{unicyclic}.  When $r_T$ \revs\ are in
$T$, exactly $\FL{r_T/m}$ tree spies will have followed them.  By
Theorem~\ref{tree}, the \revs\ never make an unguarded meeting outside $C$.

The $s-t$ unreserved spies always occupy distinct vertices of $C$; call them
``cycle spies''.  To draw $t'$ tree spies off $C$, at least $mt'$ \revs\ must
have left $C$.  Since $s=\FL{r/m}$, fewer than $m(s-t'+1)$ \revs\ remain on $C$.
Hence at most $s-t'$ meetings on $C$ need to be guarded.  Since $\ell\le s-t+2$,
we have seen in Theorem~\ref{fuzzy-cycle} that the $s-t$ cycle spies can move
to guard any desired set of $s-t$ vertices on $C$.

If at least $\ell-(s-t)$ vertices of $C$ retain tree spies, then the cycle
spies can guard the remaining vertices of $C$.  If no vertices of $C$ retain
tree spies, then $t'=t$.  In this case fewer than $m(s-t+1)$ \revs\ are on $C$,
they make at most $s-t$ meetings, and the cycle spies can guard those meetings.

Hence we may assume that $\ell=s-t+2$ and that all tree spies on $C$ are at
one vertex, $v$.  The cycle spies can guard all the other meetings on $C$
unless there are $s-t+1$ such meetings.  If there is also a meeting at $v$,
then there are at least $m(s-t+2)$ \revs\ on $C$.  Hence $m(s-t+2)<m(s-t'+1)$,
which yields $t'<t-1$.  Thus at least two tree spies remain at $v$, and one of
them can move to guard a meeting at a neighbor $u$ of $v$ on $C$.  If there is
no meeting at $v$, then again a tree spy from $v$ can guard a meeting at $u$.

This is the only way a tree spy leaves its reserved subgraph; when the
condition ends the spy moves back.  It cannot be pulled into $T$ (two steps) at
the same time, because having $m(s-t+1)$ \revs\ on $V(C)-\{v\}$ leaves fewer
than $mt$ \revs\ in the union of the trees and $v$:  not enough to pull the
last tree spy back into its tree.

\medskip
{\bf Case 2: $s\le t$}.
In this case, $\ell=3$, and the vertices of $C$ are pairwise adjacent.
Deleting the edges of $C$ leaves three disjoint trees, rooted at the vertices
of $C$.  As usual, the spies may assume that the \revs\ all start on $V(C)$.
Always, an initial position can be defended by $\FL{r/m}$ spies, and in this
case they are all on the cycle.

At a given time, let $r_i$ be the number of \revs\ on the component $T_i$ of
$G-E(C)$ (rooted at $v_i\in V(C)$), for $i\in\{1,2,3\}$.  The spies maintain
that at the end of each round there are at least $\FL{r_i/m}$ spies on tree
$T_i$, arranged according to the strategy of Theorem~\ref{tree}, with any extra
spies located at the root $v_i$.  Since each vertex appears in some $T_i$,
ability to maintain this invariant completes the proof.  The invariant holds
after the initial placement, since \revs\ appear only at the roots.

After the \revs\ move, the spies update their position on each $T_i$ via the
strategy in the proof of Theorem~\ref{tree}.  The update starts from the leaves
and works toward the root.  The \revs\ now at non-root vertices of $T_i$ were
in $T_i$ at the end of the previous round, since \revs\ enter or leave $T_i$
only via edges of $C$.    As shown in Theorem~\ref{tree}, the updates to all
vertices except the root can be completed using spies that were in $T_i$ before
the round, since the invariant held at that time.  Also, spies who moved to
$v_i$ during this process are covering \revs\ who came there from $T_i-v_i$.

Let $s_i=\FL{r_i/m}$.  If not enough \revs\ arrive at $v_i$ from the rest of
$C$ on this round to push the number of \revs\ on $T_i$ up to $ms_i+m$, then
the invariant already holds on $T_i$.  However, if $T_i$ now contains at least
$ms_i+mk$ \revs\ (for some positive k), then the number of revolutionaries
remaining in the other trees is at most $r-m(s_i+k)$, so the number of spies
needed on those trees is at most $s-s_i-k$. That is, $k$ spies are freed to
move to $v_i$.  Furthermore, since the new \revs\ in $T_i$ came from the
other roots and the tree strategy was followed using $\FL{r_j/m}$ spies on each
$T_j$, the freed spies were at the other roots and are now available to move to
$v_i$.  Doing so restores the desired invariant.
\end{proof}

The technique of Case $2$ above does not work in Case $1$, since the \revs\
can make the vertex on $C$ needing extra spies be far from the vertex with
freed spies.

We have now determined the winner for every game $\RSG$ such that $G$ is
unicyclic, and we have provided a constructive strategy for the winner in
each case.


\baselineskip 13pt
\vspace{-1pc}

{\small
\begin{thebibliography}{9}
\frenchspacing

\bibitem{BCPWZ}
J.V. Butterfield, D.W. Cranston, G.J. Puleo, D.B. West, and R. Zamani,
Revolutionaries and spies: Spy-good and spy-bad graphs, submitted.

\vspace{-.1pc}
\bibitem{HS}
D. Howard and C.D. Smyth, Revolutionaries and spies on grid-like graphs,
to appear in Discrete Math.

\end{thebibliography}
}
\end{document}